\documentclass[10pt]{amsart}
\usepackage[english]{babel}
\usepackage[T1]{fontenc}
\usepackage[latin1]{inputenc}
\usepackage{lmodern}
\usepackage{a4}
\usepackage[all]{xy}
\usepackage{verbatim}                            %
\usepackage{stmaryrd}
\usepackage{amsfonts,amsmath,amssymb,amsthm}
\usepackage{enumerate}
\usepackage{xspace}
\usepackage{ulem}
\usepackage{epsfig}
\usepackage{txfonts}
\usepackage{pxfonts}
\usepackage{babel,indentfirst}
\usepackage{euscript}                             %
\newcommand{\K}{\Bbb K}

\newcommand{\Z}{\Bbb Z}
\newcommand{\N}{\Bbb N}

\def\NL{\hfill\break}

\newcommand{\set}[1]{\left\{#1\right\}}
\newcommand{\parth}[1]{\left(#1\right)}

\makeatletter \@addtoreset{equation}{section}

\newlength{\defbaselineskip}
\newcommand{\setlinespacing}[1]%
{\setlength{\baselineskip}{#1 \defbaselineskip}}
\newtheorem{theorem}{Theorem}[section]
\newtheorem{corollary}[theorem]{Corollary}
\newtheorem{lemma}[theorem]{Lemma}
\newtheorem{proposition}[theorem]{Proposition}
\theoremstyle{remark}

\newtheorem{definition}[theorem]{Definition}
\newtheorem{example}[theorem]{Example}
\newtheorem{question}[theorem]{Question}

\begin{document}

\setlinespacing{1.1}
\title{Generalized rigid modules and their polynomial extensions}

\author[M. Louzari]{Mohamed Louzari}
\address{Mohamed Louzari\\
University Abdelmalek Essaadi, Faculty of sciences\\
Department of mathematics, Tetouan, Morocco}
\email{mlouzari@yahoo.com}

\author[A. Reyes]{Armando Reyes}
\address{Armando Reyes\\
Universidad Nacional de Colombia, Sede Bogot\'{a}\\
Departamento de Matem\'{a}ticas, Colombia}
\email{mareyesv@unal.edu.co}

\maketitle

\begin{abstract}Let $R$ be a ring with unity, $\sigma$ an endomorphism of $R$ and $M_R$ a right $R$-module. In this paper, we continue studding $\sigma$-rigid modules that were introduced by Gunner et al. \cite{generalized/rigid}. We give some results on $\sigma$-rigid modules and related concepts. Also, we study the transfer of $\sigma$-rigidness from a module $M_R$ to its extensions, as triangular matrix modules and polynomial modules.
\end{abstract}

\footnote[0]{\NL 2010 Mathematics Subject Classification. 16U80, 16S36
\NL Key words and phrases. Generalized rigid modules, Rigid modules, Reduced modules, Semicommutative modules}

\section*{Introduction }

In this paper, $R$ denotes an associative ring with unity and modules are unitary. We write $M_R$ to mean that $M$ is
a right module. Let $\N$ be the set of all natural integers. Throughout, $\sigma$ is an endomorphism of $R$ with $\sigma(1)=1$. The set of all endomorphisms (respectively, automorphisms) of $R$ is denoted by $End(R)$ (respectively, $Aut(R)$). For a subset $X$ of a module $M_R$, $r_R(X)=\{a\in R|Xa=0\}$ and $\ell_R(X)=\{a\in R|aX=0\}$ will stand for the right and the left annihilator of $X$ in $R$ respectively. A ring $R$ is called {\it semicommutative} if for every $a\in R$, $r_R(a)$ is an ideal of $R$ (equivalently, for any $a,b\in R$, $ab=0$ implies $aRb=0$). In \cite{rege2002}, a module $M_R$ is semicommutative, if for any $m\in M$ and $a\in R$, $ma=0$ implies $mRa=0$. A module $M_R$ is called $\sigma$-semicommutative \cite{zhang/chen} if, for any $m\in M$ and $a\in R$, $ma=0$ implies $mR\sigma(a)=0$.
\smallskip
\par
\smallskip
Following Annin \cite{annin}, a module $M_R$ is $\sigma$-{\it compatible}, if for any $m\in M$ and $a\in R$, $ma=0$ if and only if $m\sigma(a)=0$. An $\sigma$-compatible module $M$ is called $\sigma$-reduced if for any $m\in M$ and $a\in R$, $ma=0$ implies $mR\cap Ma=0$. The module $M_R$ is called reduced if it is $id_R$-reduced, where $id_R$ is the identity endomorphism of $R$ (equivalently, for any $m\in M$ and $a\in R$, $ma^2=0$ implies $mRa=0$) \cite{lee/zhou}.
\smallskip
\par
\smallskip
According to Krempa \cite{krempa}, an endomorphism $\sigma$ of a ring $R$ is called {\it rigid}, if $a\sigma(a)=0$ implies $a=0$ for any $a\in R$, and $R$ is called a $\sigma$-{\it rigid} ring if the endomorphism $\sigma$ is rigid. Motivated by the properties of $\sigma$-rigid rings that have been studied in \cite{hirano/99,hong/2000, hong/2003,krempa}, Guner el al. \cite{generalized/rigid}, introduced $\sigma$-rigid modules as a generalization of $\sigma$-rigid rings. A module $M_R$ is called $\sigma$-{\it rigid}, if $ma\sigma(a)=0$ implies $ma=0$ for any $m\in M$ and $a\in R$. Clearly, $\sigma$-reduced modules are $\sigma$-rigid, but the converse need not be true \cite[Example 2.18]{generalized/rigid}. Recall that, a module $M_R$ is $\sigma$-reduced if and only if it is $\sigma$-semicommutative and $\sigma$-rigid \cite[Theorem 2.16]{generalized/rigid}. Thus, the concept of $\sigma$-reduced modules is connected to the $\sigma$-semicommutative and $\sigma$-rigid modules. A module $M_R$ is called {\it rigid}, if it is $id_R$-rigid.

\smallskip
\par
\smallskip

In \cite{lee/zhou}, Lee-Zhou introduced the following notations.
$$M[x;\sigma]:=\set{\sum_{i=0}^sm_ix^i:s\geq 0,m_i\in M},$$
$$\;\;\;M[x,x^{-1};\sigma]:=\set{\sum_{i=-s}^tm_ix^i:\;t\geq 0,s\geq 0,m_i\in M}.$$

\smallskip
\par
\smallskip

Each of these is an abelian group under an obvious addition operation, and thus $M[x;\sigma]$ becomes a module over
$R[x;\sigma]$ under the following scalar product operation. For $m(x)=\sum_{i=0}^n m_ix^i\in M[x;\sigma]$ and $f(x)=\sum_{j=0}^m a_jx^j\in R[x;\sigma]$,
$$m(x)f(x)=\sum_{k=0}^{n+m}\parth{\sum_{k=i+j}m_i\sigma^i(a_j)}x^k\leqno(*)$$

\smallskip
\par
\smallskip

The module $M[x;\sigma]$ is called the {\it skew polynomial extension} of $M$. If $\sigma\in Aut(R)$, then with a scalar product similar to $(*)$, $M[x,x^{-1};\sigma]$ becomes a module over $R[x,x^{-1};\sigma]$, this is called the {\it  skew Laurent polynomial extension} of $M$. In particular, if $\sigma=id_R$ then we get the classical extensions $M[x]$ and $M[x,x^{-1}]$.

\smallskip
\par
\smallskip

According to Zhang and Chen \cite{zhang/chen}, a module $M_R$ is called $\sigma$-{\it skew Armendariz}, if $m(x)f(x)=0$ where $m(x)=\sum_{i=0}^nm_ix^i\in M[x;\sigma]$ and $f(x)=\sum_{j=0}^ma_jx^j\in R[x;\sigma]$ implies $m_i\sigma^i(a_j)=0$ for all $i$ and $j$. Also, $M_R$ is an
Armendariz module if and only if it is $id_R$-skew Armendariz. A ring $R$ is skew-Armendariz if and only if $R_R$ is a skew-Armendariz module.

\smallskip
\par
\smallskip

In this work, we study some connections between rigid, $\sigma$-rigid, semicommutative, $\sigma$-semicommutative, abelian and $\sigma$-reduced modules. Also, we show that the class of $\sigma$-rigid modules is not closed under homomorphic images and some module extensions. Moreover, we examine the transfer of $\sigma$-reducibly, ${\sigma}$-semicommutative and $\sigma$-rigidness from a module $M_R$ to its extensions of the form $M[x]/M[x](x^n)$ where $n\geq 2$ is an integer, and vice versa. Indeed, if $\;M[x]/M[x](x^n)$ is $\overline{\sigma}$-reduced (respectively, $\overline{\sigma}$-rigid, $\overline{\sigma}$-semicommutative) as a right $\;R[x]/(x^n)$-module, then $M_R$ is $\sigma$-reduced (respectively, $\sigma$-rigid, $\sigma$-semicommutative). However, the converse is not true. Furthermore, for a module $M_R$, we study the behavior of $\sigma$-rigidness regarding to its skew polynomial extensions.

\section{Generalized rigid modules and related modules}

In this section, we give some connections between rigid, $\sigma$-rigid, semicommutative, $\sigma$-semicommutative and abelian modules. We begin with the next definition.

\begin{definition}\label{df1}Let $M_R$ be a module, $m\in M$ and $a\in R$. We say that $M_R$ satisfies the condition:
\NL$(1)$  $\;(\mathcal{C}_1)$, if $ma=0$ implies $m\sigma(a)=0$.
\NL$(2)$ $\;(\mathcal{C}_2)$, if $m\sigma(a)=0)$ implies $ma=0$.
\end{definition}

\begin{lemma}\label{lem2}$(1)$ Every rigid $($respectively, $\sigma$-semicommutative$)$ module satisfying $(\mathcal{C}_2)$ is $\sigma$-rigid $($respectively, semicommutative$)$.
\NL$(2)$ Every $\sigma$-rigid $($respectively, semicommutative$)$ module satisfying $(\mathcal{C}_1)$ is rigid $($respectively, $\sigma$-semicommutative$)$.
\NL$(3)$ For $\sigma$-compatible modules, the notions of rigidity $($respectively, semicommutativity$)$ and $\sigma$-rigidity $($respectively, $\sigma$-semicommutativity$)$ coincide.
\end{lemma}

\begin{proof}The verification is straightforward.
\end{proof}

A module $M$ is called {\it abelian} \cite{agayev/2009.}, if for any $m\in M$ and any $e^2=e,r\in R$, we have $mer=mre$.

\begin{lemma}Let $M_R$ be a module and $\sigma$ an endomorphism of $R$. Then
\NL$(1)$ If $M_R$ is rigid, then it is abelian,
\NL$(2)$ If $M_R$ is $\sigma$-rigid which satisfies the condition $(\mathcal{C}_1)$, then it is abelian,
\NL$(3)$ If $M_R$ is $\sigma$-semicommutative, then it is abelian.
\end{lemma}

\begin{proof}$(1)$ Let $e^2=e,r\in R$ and $m\in M$, we have $[er(1-e)]^2=0$ then $m[er(1-e)]^2=0$, since $M_R$ is rigid then $m[er(1-e)]=0$ which gives $mer=mere$. On the other side, we have $[(1-e)re]^2=0$, with the same manner as above, it gives $mre=mere$. Thus $mer=mre$ for all $r\in R$.
\NL$(2)$ If $M_R$ is $\sigma$-rigid with the condition $(\mathcal{C}_1)$, then it is rigid by Lemma \ref{lem2}(2).
\NL$(3)$ Suppose that $M_R$ is $\sigma$-semicommutative, then $M_R$ satisfies the condition $(\mathcal{C}_1)$. Let $m\in M$ and $e^2=e\in R$, we have $me(1-e)=0$ then $meR\sigma((1-e))=0$, so $mer(1-e)=0$ for all $r\in R$, by \cite[Lemma 2.9]{lee/zhou}. Also, we have $m(1-e)e=0$ which gives $m(1-e)re=0$ for all $r\in R$. Therefore $mer=mre$ for all $r\in R$.
\end{proof}

The class of $\sigma$-rigid modules need not to be closed under homomorphic images.

\begin{example}Consider $R=\Z$, the ring of all integer numbers. Take $M=\Z_{\Z}$ and $N=12\Z$ a submodule of $M$. It is clear that $M$ is rigid, however $M/N$ is not rigid. Consider $m=3+N\in M/N$ and $a=2\in R$. We have $ma^2=0$ and $ma\neq 0$.
\end{example}

Recall that a module $M_R$ is said to be {\it torsion-free} if for every non zero-divisor $a\in R$ and $0\neq m\in M$, we have $ma\neq 0$. A submodule $N$ of a module $M$ is called a {\it prime submodule} of $M$ if whenever $ma\in N$ for $m\in M$ and $a\in R$, then $m\in N$ or $Ma\subseteq N$.

\begin{proposition}\label{prop2} Let $M_R$ be a module over a ring $R$ that has no zero divisors, and $N$ be a prime submodule of $M$ such that $M/N$ is torsion-free. Then $M$ is $\sigma$-rigid if and only if $N$ and $M/N$ are $\sigma$-rigid.
\end{proposition}

\begin{proof}$(\Rightarrow)$. Clearly, $N$ is $\sigma$-rigid. Now, let $\overline{m}\in M/N$ such that $\overline{m}a\sigma(a)=0$, hence $ma\sigma(a)\in N$. Since $N$ is prime then $m\in N$ or $Ma\sigma(a)\subseteq N$. If $m\in N$ then $ma\in N$, so $\overline{m}a=0$. Now, if $Ma\sigma(a)\subseteq N$, then $(M/N)a\sigma(a)=0$ in $M/N$. Since $M/N$ is torsion-free and $R$ without zero-divisors, we get $a\sigma(a)=0$. Then $ma\sigma(a)=0$, which implies $ma=0$ because $M$ is $\sigma$-rigid. Thus $\overline{m}a=0$. Hence $M/N$ is $\sigma$-rigid.
\NL$(\Leftarrow)$. Let $m\in M$ and $a\in R$ such that $ma\sigma(a)=0$. If $m\in N$, we have $ma=0$ because $N$ is $\sigma$-rigid. Now, suppose that $m\not\in N$. Since $N$ is prime, then from $ma\sigma(a)\in N$, we get $Ma\sigma(a)\subseteq N$. Hence $(M/N)a\sigma(a)=0$ in $M/N$. But $M/N$ is $\sigma$-rigid, then $(M/N)a=0$ which implies $a=0$, since $M/N$ is torsion-free and $R$ has no zero divisors. Therefore $ma=0$.
\end{proof}

The next example shows that the condition ``$R$ has no zero divisors'' is not superfluous in Proposition \ref{prop2}.

\begin{example}\label{example/eleborated}{\rm Let $F$ be a field,
$R=\left(
\begin{array}{cc}
F & F \\
0 & F \\
\end{array}
\right)$ and $\sigma$ an endomorphism of $R$ such that $\sigma\parth{\left(
\begin{array}{cc}
a & b \\
0 & c \\
\end{array}
\right)}=\left(
\begin{array}{cc}
a & -b \\
0 & \;\;c \\
\end{array}
\right)$. Consider the right $R$-module $M=\left(
\begin{array}{cc}
0 & F \\
F & F \\
\end{array}
\right)$ and the submodule $K=\left(
\begin{array}{cc}
0 & F \\
0 & F \\
\end{array}
\right)$.
\NL $\mathbf{(1)}$ By \cite[Example 2.6]{generalized/rigid}, $K$ and $M/K$ are $\sigma$-rigid but $M$ is not $\sigma$-rigid.
\NL$\mathbf{(2)}$ $K$ is a prime submodule of $M$. Let
$m=\left(
\begin{array}{cc}
0 & \beta \\
\alpha & \gamma \\
\end{array}
\right)\in M$ and $a=\left(
\begin{array}{cc}
x & y \\
0 & z \\
\end{array}
\right)\in R$. We have
$ma=\left(
\begin{array}{cc}
0 & \beta z\\
\alpha x & \alpha y+\gamma z \\
\end{array}
\right)$. Then $ma\in K\Leftrightarrow \alpha=0$ or $x=0$. If $\alpha=0$, we have $m\in K$. Now, suppose that $\alpha\neq 0$, so $x=0$. We can easily see that  $Ma\subseteq K$. Thus $K$ is prime.}
\NL $\mathbf{(3)}$ The ring $R$ has zero divisors. For $a=\left(
\begin{array}{cc}
0 & \;\;1 \\
0 & -1 \\
\end{array}
\right)\in R$, we can take any element of $R$ of the form $b=\left(
\begin{array}{cc}
\alpha & \alpha \\
0 & 0 \\
\end{array}
\right)$, with $\alpha\neq 0$. We have $ab=ba=0$. Hence, $a$ is a divisor of zero.

\end{example}

\section{Triangular matrix modules and polynomial modules}

In this section, we observe the $\sigma$-rigidness of some module extensions, as triangular matrix modules, ordinary polynomial modules and skew polynomial modules. The class of $\sigma$-rigid modules need not be closed under module extensions by Examples \ref{exp2} and \ref{exp vn}.

\bigskip

For a nonnegative integer $n\geq 2$. Consider

$$V_n(M):=\set{\left(%
\begin{array}{cccccc}
  m_0 & m_1 & m_2 & m_3 & \ldots & m_{n-1} \\
  0 & m_0 & m_1 & m_2 & \ldots & m_{n-2} \\
  0 & 0 & m_0 & m_1 & \ldots & m_{n-3}\\
  \vdots & \vdots &\vdots& \vdots & \ddots &\vdots \\
  0 & 0 & 0 & 0 & \ldots & m_1 \\
  0 & 0 & 0 & 0& \ldots & m_0 \\
\end{array}%
\right)\bigm| m_0,m_1,m_2,\cdots,m_{n-1}\in M}$$
and

$$V_n(R):=\set{\left(%
\begin{array}{cccccc}
  a_0 & a_1 & a_2 & a_3 & \ldots & a_{n-1} \\
  0 & a_0 & a_1 & a_2 & \ldots & a_{n-2} \\
  0 & 0 & a_0 & a_1 & \ldots & a_{n-3}\\
  \vdots & \vdots &\vdots& \vdots & \ddots &\vdots \\
  0 & 0 & 0 & 0 & \ldots & a_1 \\
  0 & 0 & 0 & 0& \ldots & a_0 \\
\end{array}%
\right)\bigm| a_0,a_1,a_2,\cdots,a_{n-1}\in R}$$

\bigskip

We denote elements of $V_n(R)$ by $(a_0,a_1,\cdots,a_{n-1})$, and elements of $V_n(M)$ by $(m_0,m_1,\cdots,m_{n-1})$. Clearly, $V_n(M)$ is a right $V_n(R)$-module under the usual matrix addition operation and the following scalar product operation.

\smallskip

For $U=(m_0,m_1,\cdots,m_{n-1})\in V_n(M)$ and $A=(a_0,a_1,\cdots,a_{n-1})\in V_n(R)$, $UA=(m_0a_0,m_0a_1+m_1a_0,\cdots,m_0a_{n-1}+m_1a_{n-2}+\cdots+m_{n-1}a_0)\in V_n(M)$. Furthermore, $V_n(M)\cong M[x]/M[x](x^n)$ where $M[x](x^n)$ is a submodule of $M[x]$ generated by $x^n$ and $V_n(R)\cong R[x]/(x^n)$ where $(x^n)$ is an ideal of $R[x]$ generated by $x^n$. The endomorphism $\sigma$ of $R$ can be extended to $V_n(R)$ and $R[x]$, and we will denote it in both cases by $\overline{\sigma}$. From now on, we will freely use the isomorphisms $$V_n(M)\cong M[x]/M[x](x^n)\;\mathrm{and}\;V_n(R)\cong R[x]/(x^n).$$

\begin{proposition}\label{lem1}Let $M_R$ be a module and an integer $n\geq 2$. Then
\NL$(1)$ If $M[x]/M[x](x^n)$ is $\overline{\sigma}$-semicommutative $($as a right $R[x]/(x^n)$-module$)$, then $M_R$ is $\sigma$-semicommutative,
\NL$(2)$ If $M[x]/M[x](x^n)$ is $\overline{\sigma}$-rigid $($as a right $R[x]/(x^n)$-module$)$, then $M_R$ is $\sigma$-rigid.
\end{proposition}

\begin{proof}Let $U=(m,0,0,\cdots,0)\in V_n(M)$ and $A=(a,0,0,\cdots,0)\in V_n(R)$ with $m\in M_R$ and $a\in R$.
\NL$(1)$ Suppose that $ma=0$. Then $MA=(ma,0,0,\cdots,0)=0$, since $V_n(M)$ is $\overline{\sigma}$-semicommutative, we have $UB\overline{\sigma}(A)=0$ for any $B=(r,0,\cdots,0)\in V_n(R)$ with $r$ be an arbitrary element of $R$. But $UB\overline{\sigma}(A)=(mr\sigma(a),0,\cdots,0)=0$, hence $mr\sigma(a)=0$ for any $r\in R$. Therefore $M_R$ is $\sigma$-semicommutative.
\NL$(2)$ Suppose that $ma\sigma(a)=0$. Then $UA\overline{\sigma}(A)=(ma\sigma(a),0,0,\cdots,0)=0$. Since $V_n(M)_{V_n(R)}$ is $\overline{\sigma}$-rigid, we get $UA=0=(ma,0,0,\cdots,0)$, so $ma=0$. Therefore $M_R$ is $\sigma$-rigid.
\end{proof}

\begin{corollary}\label{cor8}Let $M_R$ be a module and $n\geq 2$ an integer. If $M[x]/M[x](x^n)$ is $\overline{\sigma}$-reduced $($as a right $R[x]/(x^n)$-module$)$, then $M_R$ is $\sigma$-reduced.
\end{corollary}

\begin{corollary}[{\cite[Proposition 2.14]{agayev2007}}] If the ring $R[x]/(x^n)$ is semicommutative for any integer $n\geq 2$. Then R is a semicommutative ring.
\end{corollary}

\par All reduced rings are Armendariz. There are non-reduced Armendariz rings. For a positive integer $n$, let $\Z/n\Z$ denotes ring of integers modulo $n$. For each positive integer $n$, the ring $\Z/n\Z$ is Armendariz but not reduced whenever $n$ is not square free  \cite[Proposition 2.1]{chhawchharia}.

\begin{corollary}\label{cor4} If $M[x]/M[x](x^n)$ is $\overline{\sigma}$-reduced as a right $R[x]/(x^n)$-module, then $M_R$ is $\sigma$-skew Armendariz.
\end{corollary}

\begin{proof} If $M_R$ is $\sigma$-reduced then it is $\sigma$-skew Armendariz (see \cite[Theorem 2.20]{agayev2/2009}).
\end{proof}

Note that, the converse of Corollary \ref{cor4} need not be true, since $M=\Z/n\Z$ is Armendariz but not reduced ($n$ is not square free) and so $V_n(M)$ is not reduced, by Corollary \ref{cor8}. Furthermore, the converse of Proposition \ref{lem1} and Corollary \ref{cor8} may not be true by the next examples. For the first one, if $M_R$ is $\sigma$-semicommutative then the right $R[x]/(x^n)$-module $M[x]/M[x](x^n)$ need not to be $\overline{\sigma}$-semicommutative.

\begin{example}[{\cite[Example 2.9]{baser/2008}}]\label{exp3} Consider the ring $R=\set{\begin{pmatrix}
  a & b \\
  0 & a
\end{pmatrix}| a,b\in \Z}$, with an endomorphism $\sigma$ defined by $\sigma\left(\begin{pmatrix}
  a & b \\
  0 & a
\end{pmatrix}\right)=\begin{pmatrix}
  a & -b \\
  0 & \;\;a
\end{pmatrix}$. The module $R_R$ is $\sigma$-semicommutative by {\cite[Example 2.5(1)]{baser/2008}}. Take $$A=\left(
                          \begin{array}{cc}
                            \left(
                              \begin{array}{cc}
                                0 & 1 \\
                                0 & 0 \\
                              \end{array}
                            \right)
                             & \left(
                                 \begin{array}{cc}
                                   -1 & \;\;1 \\
                                   \;\;0 & -1 \\
                                 \end{array}
                               \right)
                              \\
                            \left(
                              \begin{array}{cc}
                                0 & 0 \\
                                0 & 0 \\
                              \end{array}
                            \right)
                             & \left(
                                 \begin{array}{cc}
                                   0 & 1 \\
                                   0 & 0 \\
                                 \end{array}
                               \right)
                              \\
                          \end{array}
                        \right),\;B=\left(
                          \begin{array}{cc}
                            \left(
                              \begin{array}{cc}
                                0 & 1 \\
                                0 & 0 \\
                              \end{array}
                            \right)
                             & \left(
                                 \begin{array}{cc}
                                   1 & 1 \\
                                   0 & 1 \\
                                 \end{array}
                               \right)
                              \\
                            \left(
                              \begin{array}{cc}
                                0 & 0 \\
                                0 & 0 \\
                              \end{array}
                            \right)
                             & \left(
                                 \begin{array}{cc}
                                   0 & 1 \\
                                   0 & 0 \\
                                 \end{array}
                               \right)
                              \\
                          \end{array}
                        \right)\in V_2(R).$$
We have $AB=0$, but for $C=\left(
                          \begin{array}{cc}
                            \left(
                              \begin{array}{cc}
                                1 & 0 \\
                                0 & 1 \\
                              \end{array}
                            \right)
                             & \left(
                                 \begin{array}{cc}
                                   0 & 0 \\
                                   0 & 0 \\
                                 \end{array}
                               \right)
                              \\
                            \left(
                              \begin{array}{cc}
                                0 & 0 \\
                                0 & 0 \\
                              \end{array}
                            \right)
                             & \left(
                                 \begin{array}{cc}
                                   1 & 0 \\
                                   0 & 1 \\
                                 \end{array}
                               \right)
                              \\
                          \end{array}
                        \right)\in V_2(R)$, we get $AC\overline{\sigma}(B)=\left(
                          \begin{array}{cc}
                            \left(
                              \begin{array}{cc}
                                0 & 0 \\
                                0 & 0 \\
                              \end{array}
                            \right)
                             & \left(
                                 \begin{array}{cc}
                                   0 & 2 \\
                                   0 & 0 \\
                                 \end{array}
                               \right)
                              \\
                            \left(
                              \begin{array}{cc}
                                0 & 0 \\
                                0 & 0 \\
                              \end{array}
                            \right)
                             & \left(
                                 \begin{array}{cc}
                                   0 & 0 \\
                                   0 & 0 \\
                                 \end{array}
                               \right)
                              \\
                          \end{array}
                        \right)\neq 0$. Thus, $V_2(R)$ is not $\overline{\sigma}$-semicommutative as a right $V_2(R)$-module.
\end{example}

\smallskip

Now, if $M_R$ is $\sigma$-reduced (respectively, $\sigma$-rigid) then the right $R[x]/(x^n)$-module $M[x]/M[x](x^n)$ is not $\overline{\sigma}$-reduced (respectively, $\overline{\sigma}$-rigid).

\smallskip

\begin{example}\label{exp2}Let $\K$ be a field, then $\K$ is reduced as a right $\K$-module. Consider the polynomial $p=\overline{x}\in \K[x]/x^2\K[x]$. We have $p\neq \overline{0}$, because $x\not\in x^2\K[x]$ but $p^2=\overline{0}$. Therefore, $\K[x]/x^2\K[x]\simeq V_2(\K)$ is not reduced as a right $V_2(\K)$-module
\end{example}

\begin{example}\label{exp vn}Let $R=\K\langle x,y\rangle$ be the ring of polynomials in two non-commuting indeterminates over a field $\K$. Consider the right $R$-module $M=R/xR$. The module $M$ is rigid but not semicommutative by \cite[Example 2.18]{generalized/rigid}.
\NL Now, for $U=\left(
\begin{array}{cc}
1+xR & xR \\
0 & 1+xR \\
\end{array}
\right)\in V_2(M)$ and
$A=\left(
\begin{array}{cc}
0 & 1 \\
0 & 0 \\
\end{array}
\right)\in V_2(R)$. We have $UA^2=0$, however $UA=\left(
\begin{array}{cc}
0 & 1+xR \\
0 & 0 \\
\end{array}
\right)\neq 0$. Thus $V_2(M)$ is not rigid as a right $V_2(R)$-module.
\end{example}

\bigskip

Let $n\geq 2$ be an integer, we have:

$$
\begin{array}{c}
  V_n(M)\; is \;\overline{\sigma}-rigid \\
   \not\Uparrow\;\;\Downarrow  \\
 M_R\; is \;\sigma-rigid \\
  \not\Downarrow\;\;\Uparrow \\
 M[x]\; is \;\overline{\sigma}-rigid
\end{array}
$$

\bigskip

\begin{question}Under which conditions we have $M_R$ is $\sigma$-rigid $\Rightarrow$ $V_n(M)$ is $\overline{\sigma}$-rigid?
\end{question}

\begin{proposition}\label{prop4}Let $n\geq 2$ be an integer and $M_R$ a $\sigma$-rigid module. If $ M_R$ is $\sigma$-semicommutative, then $M[x]/M[x](x^n)$ is $\overline{\sigma}$-semicommutative as a right $R[x]/(x^n)$-module.
\end{proposition}

\begin{proof} Let $U=(m_0,m_1,\cdots,m_{n-1})\in V_n(M)$ and $A=(a_0,a_1,\cdots,a_{n-1}),$ $B=(b_0,b_1,\cdots,b_{n-1})\in V_n(R)$ such that $UA=0$. We will show $UB\overline{\sigma}(A)=0$. We have
$UA=(m_0a_0,m_0a_1+m_1a_0,\cdots,m_0a_{n-1}+\cdots+m_{n-1}a_0)$, $UB=(m_0b_0,m_0b_1+m_1b_0,\cdots,m_0b_{n-1}+\cdots+m_{n-1}b_0)=(\beta_0,\beta_1,\cdots,\beta_{n-1})$ and
$UB\overline{\sigma}(A)=(\beta_0\sigma{(a_0)},\beta_0\sigma{(a_1)}+\beta_1\sigma{(a_0)},\cdots, \beta_0\sigma{(a_{n-1})}+\beta_1\sigma{(a_{n-2})}+\cdots+\beta_{n-1}\sigma(a_0))$. From $UA=0$, we get the following system of equations.
$$\;0=m_0a_0\qquad\qquad\qquad\qquad\qquad\qquad\quad\leqno(0)$$
$$=m_0a_1+m_1a_0\qquad\qquad\qquad\qquad\;\;\leqno(1)$$
$$\vdots\qquad\qquad\qquad\qquad$$
$$\;=m_0a_{n-1}+m_1a_{n-2}+\cdots+m_{n-1}a_0\leqno(n-1)$$

\NL Equation (0) implies $m_0a_0=0$, then $m_0R\sigma(a_0)=0$. Multiplying Equation (1) on the right side by $\sigma(a_0)$, we get $m_0a_1\sigma(a_0)+m_1a_0\sigma(a_0)=0\;\;(1')$. Since $m_0a_1\sigma(a_0)=0$ because $m_0R\sigma(a_0)=0$, we have $m_1a_0\sigma(a_0)=0$. But $M_R$ is $\sigma$-rigid then $m_1a_0=0$. From Equation (1) we have also $m_0a_1=0$. For both cases, we have $m_1R\sigma(a_0)=m_0R\sigma(a_1)=0$. Summarizing at this point: $m_0a_0=m_0a_1=m_1a_0=0$ and $m_0R\sigma(a_0)=m_0R\sigma(a_1)=m_1R\sigma(a_0)=0$. Suppose that the result is true until $i$, multiplying Equation (i+1) on the right side by $\sigma(a_0)$, we get
$$0=m_0a_{i+1}\sigma(a_0)+m_1a_{i}\sigma(a_0)+\cdots+m_{i}a_1\sigma(a_0)+m_{i+1}a_0\sigma(a_0)\leqno(i+1)'$$
By the inductive hypothesis, we have $$m_0a_{i+1}\sigma(a_0)=m_1a_{i}\sigma(a_0)=\cdots=m_{i}a_1\sigma(a_0)=0.$$
Then $m_{i+1}a_0\sigma(a_0)=0$, which gives $m_{i+1}a_0=0$ and so $m_{i+1}R\sigma(a_0)=0$. Continuing with the same manner by multiplying Equation $(i+1)'$ on the right side, respectively by $\sigma(a_1),\sigma(a_2),\cdots,\sigma(a_{i})$ we get $m_ka_{i+1-k}=0$ for all $k=0,1,\cdots, i+1$. Hence, we get the result for $i+1$. That is $m_ia_j=0$, and so $m_iR\sigma(a_j)=0$ for all integers $i$ and $j$ with $i+j\leq n-1$. Therefore, all components of $UB\overline{\sigma}(A)$ are zero. Thus $V_n(M)$ is $\overline{\sigma}$-semicommutative.
\end{proof}

\begin{corollary}\label{cor5}$(1)$ If $M_R$ is $\sigma$-reduced then $M[x]/M[x](x^n)$ is $\overline{\sigma}$-semicommutative as a right $R[x]/(x^n)$-module.
\NL$(2)$ If $M_R$ is reduced then $M[x]/M[x](x^n)$ is semicommutative as a right $R[x]/(x^n)$-module.
\NL$(2)$ If $R$ is reduced then $R[x]/(x^n)$ is semicommutative.
\end{corollary}

\begin{corollary}[{\cite[Theorem 3.9]{zhang/chen}}]\label{cor6}$(1)$ If $M_R$ is a $\sigma$-rigid module. Then $M_R$ is $\sigma$-semicommutative if and only if $M[x]/M[x](x^n)$ is $\overline{\sigma}$-semicommutative as a right $R[x]/(x^n)$-module.
\NL$(2)$ If $M_R$ is a rigid module. Then $M_R$ is semicommutative if and only if $M[x]/M[x](x^n)$ is semicommutative as a right $R[x]/(x^n)$-module.
\end{corollary}

\begin{proof} It follows from Propositions \ref{lem1}(1) and \ref{prop4}.
\end{proof}

\begin{corollary}Let $n\geq 2$ be an integer. If $M[x]/M[x](x^n)$ is Armendariz as a right $R[x]/(x^n)$-module, then $M[x]/M[x](x^n)$  is semicommutative as a right $R[x]/(x^n)$-module.
\end{corollary}

\begin{proof} It follows from \cite[Theorem 1.9]{lee/zhou} and Corollary \ref{cor5}(2).
\end{proof}

For $U=(m_0,m_1,\cdots,m_{n-1})\in V_n(M),\;A=(a_0,a_1,\cdots,a_{n-1})\in V_n(R)$.
\NL Consider $\alpha_i\in V_n(M)$, such that $\alpha_i=m_0a_i+m_1a_{i-1}+\cdots+m_ia_0$ for all $i\in\{0,1,\cdots,n-1\}$. We have $UA=(\alpha_0,\alpha_1,\cdots,\alpha_{n-1})$ and $UA\overline{\sigma}(A)=(\alpha_0\sigma(a_0),\alpha_0\sigma(a_1)+\alpha_1\sigma(a_0),\cdots,\alpha_0\sigma(a_{n-1})+\alpha_1\sigma(a_{n-2})+\cdots+\alpha_{n-1}\sigma(a_0))$.

\begin{proposition}\label{prop6}Let $n\geq 2$ be an integer and $M_R$ a $\sigma$-reduced module. If $\;UA\overline{\sigma}(A)=0$ then $\alpha_i\sigma(a_j)=0$ for all nonnegative integers $i,j$ with $i+j=0,1,\cdots,n-1$.
\end{proposition}

\begin{proof} From $UA\overline{\sigma}(A)=0$, we get the following system of equations.
$$\;0=m_0\sigma(a_0)\qquad\qquad\qquad\qquad\qquad\qquad\qquad\quad\;\;\;\leqno(0)$$
$$=m_0\sigma(a_1)+m_1\sigma(a_0)\qquad\qquad\qquad\qquad\qquad\leqno(1)$$
$$\vdots\qquad\qquad\qquad\qquad$$
$$\;=m_0\sigma(a_{n-1})+m_1\sigma(a_{n-2})+\cdots+m_{n-1}\sigma(a_0)\leqno(n-1)$$

\NL For $i+j=0$, Equation (0) implies $\alpha_0\sigma(a_0)=0$. Assume that $k\geq 0$, and suppose that $\alpha_i\sigma(a_j)$ for all $i,j$ with $i+j\leq k$. Now, multiplying Equation $(k+1)$ on the right side by $\sigma^2(a_0)$ we get:
$$\alpha_0\sigma(a_{k+1})\sigma^2(a_0)+\alpha_1\sigma(a_{k})\sigma^2(a_0)+\cdots+\alpha_{k+1}\sigma(a_{0})\sigma^2(a_0)=0\leqno(k+1)'$$
By the inductive hypothesis, we have $\alpha_i\sigma(a_0)=0$ for all $0\leq i\leq k$, then $\alpha_iR\sigma^2(a_0)=0$ for all $0\leq i\leq k$ because $M_R$ is $\sigma$-semicommutative. Thus, Equation $(k+1)'$ gives $\alpha_{k+1}\sigma(a_0)\sigma^2(a_0)=0$ which implies $\alpha_{k+1}\sigma(a_0)=0$, by the $\sigma$-rigidness of $M_R$. So, Equation $k+1$ becomes:
$$\alpha_0\sigma(a_{k+1})+\alpha_1\sigma(a_{k})+\cdots+\alpha_k\sigma(a_1)=0\leqno(k+1)''$$
Multiplying Equation $(k+1)''$ on the right side by $\sigma^2(a_1)$ and use the fact that $\alpha_iR\sigma^2(a_1)=0$ for al $0\leq i\leq k-1$, we get $\alpha_k\sigma(a_1)\sigma^2(a_1)=0$, which implies $\alpha_k\sigma(a_1)=0$. Continuing this procedure yields $\alpha_i(a_j)=0$ for all $i,j$ such that $i + j = k + 1$. Therefore, $\alpha_i\sigma(a_j)=0$ for all nonnegative integers $i,j$ with $i+j=0,1,\cdots,n-1$.
\end{proof}

\par Let $R$ be a commutative domain. The set $T(M)=\{m\in M\;|\;r_R(m)\neq 0\}$ is called the {\it torsion submodule} of $M_R$. If $T(M)=M$ (respectively, $T(M)=0$) then $M_R$ is {\it torsion} (respectively, {\it torsion-free}).

\begin{corollary}\label{cor7}Let $R$ be a commutative domain, $M_R$ a torsion-free module. Then for any integer $n\geq 0$, we have:
\NL$(1)$ $M_R$ is $\sigma$-reduced if and only if $M[x]/M[x](x^n)$ is $\overline{\sigma}$-reduced as $R[x]/(x^n)$-module.
\NL$(2)$ $M_R$ is reduced if and only if $M[x]/M[x](x^n)$ is reduced as $R[x]/(x^n)$-module.
\end{corollary}

\begin{proof}$(\Leftarrow)$ Obvious from Corollary \ref{cor8}. $(\Rightarrow)$ From Proposition \ref{prop4}, $V_n(M)$ is $\overline{\sigma}$-semicommutative. On the other hand, if we take $U$ and $A$ as in Proposition \ref{prop6} such that $UA\overline{\sigma}(A)=0$, then we get $\alpha_i\sigma(a_j)=0$ for all $i,j$, then $\alpha_i=0$ for all $i$, because $M_R$ is torsion-free and so $UA=0$. Therefore $V_n(M)$ is $\overline{\sigma}$-rigid. Hence $V_n(M)$ is $\overline{\sigma}$-reduced.
\end{proof}

A regular element of a ring $R$ means a nonzero element which is not zero divisor. Let $S$ be a multiplicatively closed subset of $R$ consisting of regular central elements. We may localize $R$ and $M$ at $S$. Let $\sigma$ be an endomorphism of $R$, consider the map $S^{-1}\sigma\colon S^{-1}R\rightarrow S^{-1}R$ defined by $S^{-1}\sigma(a/s)=\sigma(a)/s$ with $\sigma(s)=s$ for any $s\in S$. Then $S^{-1}\sigma$ is an endomorphism of the ring $S^{-1}R$. Clearly $S^{-1}\sigma$ extends $\sigma$, we will denote it by $\sigma$. In the next, we will discuss when the localization $S^{-1}M$ is $\sigma$-rigid as a right $S^{-1}R$-module.

\begin{lemma}\label{laurent/module} For a multiplicatively closed subset $S$ of a ring $R$ consisting of all central regular elements. A right $R$-module $M$ is $\sigma$-rigid if and only if $S^{-1}M$ is $\sigma$-rigid $($as a right $S^{-1}R$-module$)$.
\end{lemma}

\begin{proof}Clearly, if $S^{-1}M$ is $\sigma$-rigid then $M$ is $\sigma$-rigid, because the class of $\sigma$-rigid modules is closed under submodules. Conversely, suppose that $M_R$ is $\sigma$-rigid, let $(m/s)\in S^{-1}M$ and $(a/t)\in S^{-1}R$ such that $(m/s)(a/t)\sigma((a/t))=0$ in $S^{-1}M$. Then $ma\sigma(a)=0$, by hypothesis $ma=0$. Hence $(m/s)(a/t)=0$, so that $S^{-1}M$ is $\sigma$-rigid.
\end{proof}

\begin{proposition}\label{cor1}Let $M$ be a right $R$-module. Then $M[x]_{R[x]}$ is $\sigma$-rigid if and only if $M[x,x^{-1}]_{R[x,x^{-1}]}$ is $\sigma$-rigid.
\end{proposition}

\begin{proof}Consider $S=\{1,x,x^2,\cdots\}\subseteq R[x]$. It is clear that $S$ is a multiplicatively closed subset of the ring $R[x]$ consisting of all central regular elements. Also, we have $S^{-1}M[x]=M[x,x^{-1}]$ and $S^{-1}R[x]=R[x,x^{-1}]$. By Lemma \ref{laurent/module}, we get the result.
\end{proof}

\begin{corollary}Let $R$ be a ring and $M$ a right $R$-module. Then
\NL$(1)$  $M[x]_{R[x]}$ is rigid if and only if $M[x,x^{-1}]_{R[x,x^{-1}]}$ is rigid,
\NL$(2)$  $R[x]$ is rigid if and only if $R[x,x^{-1}]$ is rigid.
\end{corollary}

\begin{proof}Clearly from Proposition \ref{cor1}.
\end{proof}

\begin{proposition}\label{prop3}If $M[x;\sigma]_{R[x;\sigma]}$ is rigid then $M_R$ is $\sigma$-rigid.
\end{proposition}

\begin{proof}Assume that $M[x;\sigma]$ is rigid. Consider $m(x)=m\in M[x;\sigma]$ and $f(x)=ax\in R[x;\sigma]$ such that $ma\sigma(a)=0$. We have $m(x)[f(x)]^2=ma\sigma(a)x^2=0$, then $m(x)f(x)=0$ because $M[x;\sigma]$ is rigid. But $m(x)f(x)=0$ implies $ma=0$. Which completes the proof.
\end{proof}

According to Krempa \cite{krempa}, if a ring $R$ is $\sigma$-rigid then $R[x;\sigma]$ is reduced, hence $R[x;\sigma]$ is rigid. Also, if a module $M_R$ is $\sigma$-rigid and $\sigma$-semicommutative, then $M[x;\sigma]_{R[x;\sigma]}$ is reduced \cite[Theorem 1.6]{lee/zhou}. However, we do not know if we can drop the $\sigma$-semicommutativity from this implication?

\section*{acknowledgments}The authors are deeply indebted to the referee for many helpful comments and suggestions for the improvement of this paper.

\end{document}